\documentclass{amsart}

\usepackage{matlab-prettifier}
\usepackage{tikz}
\usepackage{graphicx}
\usepackage{amsthm}
\usepackage{amsmath}
\usepackage{amsfonts}
\usepackage{amssymb}
\usepackage{cite}

\newtheorem{lemma}{Lemma}
\newtheorem{corollary}{Corollary}
\newtheorem{theorem}{Theorem}
\newtheorem{oldthm}{Theorem}

\numberwithin{equation}{section}

\title[Optimizing Resource Distribution]{Optimizing Resource Distribution in a One-Dimensional Logistic Diffusion Model}
\author{Junyoung Heo}
\address{Department of Mathematical Sciences, KAIST}
\email{joeheomail@kaist.ac.kr}

\author{Yubin Lee}
\address{Department of Mathematical Sciences, KAIST}
\email{youbin0606@kaist.ac.kr}

\begin{document}

\begin{abstract}
In this article, we study the optimization of resource distributions in a one-dimensional logistic diffusive model. The goal is to determine a distribution on a bounded one-dimensional domain that maximizes the total population at equilibrium. Previous works have shown that optimal resources are bang-bang, and in one dimension, a sufficiently large dispersal rate forces the optimal resource to be concentrated. For general dispersal rates, however, the analysis becomes more difficult because the equilibrium population may behave irregularly, and the optimal resource may be fragmented. To address this, we introduce a block decomposition that reduces fragmented resources to a collection of concentrated blocks. We then define an advantage function, which measures the gain in the equilibrium population obtained by allocating resources on a fixed interval and is used to analyze the contribution of each block to the total population. This function also allows us to reformulate the optimization problem as a convexity analysis of the advantage function. We prove the superlinearity of this function when the total resource is small enough, and this property leads to an explicit characterization of the optimal control with sufficiently small total resource.
\end{abstract}
\maketitle

\section{Introduction}

\noindent\textit{1.1.\ Problem Statement}
\smallskip

We consider the logistic diffusion model for a bounded domain \(\Omega \subset \mathbb{R}\).

\begin{equation}\label{main}
\begin{cases}
\theta_t = \mu \Delta \theta + \theta (m - \theta) & \textnormal{in  } \Omega \times (0, \infty)
\\
\dfrac{\partial\theta}{\partial\nu} = 0 & \textnormal{on  } \partial\Omega \times (0, \infty),
\end{cases}
\end{equation}
where \(\theta : \Omega \to \mathbb{R}_+\) is the population density of a species, \(m \in L^{\infty}(\Omega)\) is a resource, \(\mu > 0\) is the dispersal rate, \(\partial\Omega\) is of class piece-wise \(C^3\). The following convergence and regularity result, established by Cantrell and Cosner \cite{CC91}, holds for solutions of (\ref{main}).

\begin{oldthm}
    Assume that \(m \in C^2(\Omega)\), and \(\theta|_{t=0} \not \equiv 0\) is given. Then, there exists \(\theta^* \in C^{2, \alpha}(\bar{\Omega})\), independent of the initial condition \(\theta|_{t = 0}\), such that \(\theta \to \theta^*\) uniformly as \(t \to \infty\). In addition, \(\theta^*\) is a unique solution of the following semilinear elliptic equation.
    \begin{equation}\label{equibra}
    \begin{cases}
    \mu \Delta \theta + \theta (m - \theta) = 0 & \textnormal{in  } \Omega
    \\
    \dfrac{\partial\theta}{\partial\nu} = 0 & \textnormal{on  } \partial\Omega
    \\
    \theta \geq 0, \theta \not \equiv 0
    \end{cases}
    \end{equation}
    \end{oldthm}

For general \(m \in L^{\infty}(\Omega)\), if the solution of (\ref{main}) reaches a nonnegative equilibrium state \(\theta_{m, \mu}\), then it solves (\ref{equibra}). Moreover, the existence and uniqueness of \(\theta_{m, \mu}\) are given by the following theorem, which was also established in \cite{CC91}.

\begin{oldthm}\label{uniqueness}
    If \(m \in L^\infty (\Omega)\) satisfies \(\int_\Omega m > 0\), for every \(\mu > 0\), (\ref{equibra}) has a unique weak solution \(\theta_{m, \mu} \in C^{1, \alpha} (\Omega)\).
\end{oldthm}

We define the total population for \(\mu > 0\) and \(m \in \mathcal{M}(\Omega)\) by

\begin{equation}
F_{\mu}(m):=\int_{\Omega}\theta_{m ,\mu},
\end{equation}
where \(\mathcal{M}\) is the class of resource \(m\) introduced in \cite{L08}.

\begin{equation}
\mathcal{M}(\Omega):= \left\{ m \in L^{\infty}(\Omega), 0 \leq m \leq 1, \frac{1}{|\Omega|}\int_{\Omega}m = m_0\right\}.
\end{equation}

for a fixed \(m_0 \in (0, 1)\). There are several studies \cite{BHL16, L08,CC91,MNP21,MNP20,L06} interested in the optimal total population in the class \(\mathcal{M}\) of resources.

\begin{equation}\label{opt}
\sup_{m \in \mathcal{M}(\Omega)}F_{\mu}(m).
\end{equation}

It is shown in \cite{MNP21} that there exists an optimal control \(m^{*}_\mu\) of the optimization problem by applying the calculus of variations.

\begin{oldthm}\label{exist}
There exists an optimal control \(m^{*}_\mu \in \mathcal{M}(\Omega)\) of \textnormal{(\ref{opt})}.
\end{oldthm}

In addition, in \cite{MNP21}, it was proved that a general property of \(m_\mu^*\) called bang-bang follows.

\begin{oldthm}\label{bang}
Let \(m^{*}_\mu \in \mathcal{M}(\Omega)\) be an optimal control of \textnormal{(\ref{opt})}. Then there exists a measurable subset \(E \subset \Omega\) such that
\begin{equation}
m^*_\mu = \chi_E.
\end{equation}
\end{oldthm}

Finally, the bounds of \(F_\mu(m)\) were obtained in \cite{BHL16}.

\begin{oldthm}\label{boundineq}
    For any \(m \in \mathcal{M}(\Omega)\),
    \begin{equation}
        m_0|\Omega| \leq F_\mu(m) \leq 3m_0|\Omega|.
    \end{equation}
\end{oldthm}

\smallskip
\noindent\textit{1.2.\ Main Results}
\smallskip

Since every \(m_\mu^*\) is of bang-bang type, we can find a sequence of bang-bang type resources \(\{m_n\} \subset \mathcal{M}(\Omega) \cap BV(\Omega)\) such that \(m_n \to m_\mu^*\) in \(L^2(\Omega)\) as \(n \to \infty\). Therefore, it is important to investigate a bang-bang type resource \(m \in \mathcal{M}(\Omega) \cap BV(\Omega)\) in order to identify \(m_\mu^*\).

Let \(\Omega = (a, b)\). For any bang-bang type resource \(m \in \mathcal{M}(\Omega) \cap BV(\Omega)\), assume that there exists a partition \(a = x_0 < x_1 < \cdots < x_r < x_{r+1} = b\) of \(\Omega\) which satisfies that \(\theta_{m, \mu}'(x_i) = 0\), and
\begin{equation}\label{monom}
    m(x) = \chi_{(x'_i, x_{i + 1})}(x) \text{ or } m(x) = \chi_{(x_i, x'_i)}(x) \text{ a.e.}
\end{equation}
for some \(x'_i \in (x_i, x_{i + 1}):=\Omega_i\), for each \(i = 0, 1, \cdots, r\). Then, we say the resource \(m\) and the corresponding population distribution \(\theta_{m, \mu}\) are block-decomposable. In other words, we can analyze \(\theta_{m, \mu}\) by regarding it as the composition of weak solutions of (\ref{equibra}) on subdomains \(\Omega_1, \Omega_2, \cdots, \Omega_r\), as illustrated below.

\begin{figure}[htb!]
    \centering
    \begin{tikzpicture}[scale = 3.4]
        \draw[thick, ->](-0.1, 0)--(2.2, 0) node [anchor=north]{\(x\)};
        \draw[thick, ->](0, -0.1)--(0, 0.9);
        \draw[thick](2.1, -0.025)--(2.1, 0.025);
        \draw[thick](-0.025, 0.8)--(0.025, 0.8);
        \draw[dashed](0.5, 0)--(0.5, 0.8);
        \draw[dashed](1.6, 0)--(1.6, 0.8);
        \draw[dashed](2.1, 0)--(2.1, 0.8);
        \draw[dashed](1, 0)--(1, 0.8);
        \draw[dashed](0, 0.8)--(2.1, 0.8);
        \draw[dashed, blue](0.75, 0.8)--(0.75, 0);
        \draw[dashed, blue](1.3, 0.8)--(1.3, 0);
        \draw[thick] (0.5,0.1) parabola (0.75,0.4);
        \draw[thick] (1,0.7) parabola (0.75,0.4);
        \draw[thick] (1,0.7) parabola (1.3,0.5);
        \draw[thick] (1.6,0.3) parabola (1.3,0.5);
        
        \draw[thick, blue] (0.5,0)--(0.75,0);
        \draw[thick, blue] (0.75,0.8)--(1.3,0.8);
        \draw[thick, blue] (1.3,0)--(1.6,0);

        \node[left] at (0, 0.45) {\(\theta_{m, \mu}\)};
        \node[above, blue] at (1, 0.8) {\(m\)};
        \node[centered] at (0.25, 0.4) {\Large\(\cdots\)};
        \node[centered] at (1.85, 0.4) {\Large\(\cdots\)};
        \node[below] at (2.1, -0.025) {\(1\)};
        \node[left] at (-0.025, 0.8) {\(1\)};
        \node[below] at (-0.2, 0) {\(0\)};
        \node[below] at (0.5, 0) {\(x_i\)};
        \node[below] at (0.75, 0) {\(x'_i\)};
        \node[below] at (1, 0) {\(x_{i + 1}\)};
        \node[below] at (1.3, 0) {\(x'_{i + 1}\)};
        \node[below] at (1.6, 0) {\(x_{i + 2}\)};
    \end{tikzpicture}
    \caption{Block Decomposition of \(\theta_{m, \mu}\)}
    \label{1}
\end{figure}
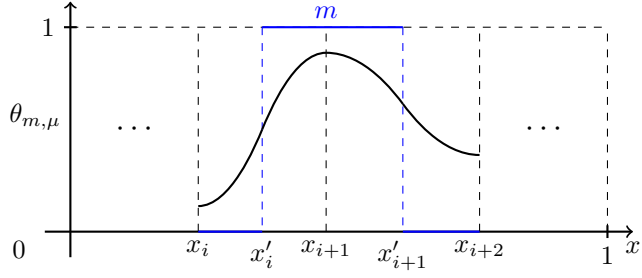

We showed that \(m\) being block decomposable is a necessary condition for \(m\) to be an optimal control of (\ref{opt}).

\begin{theorem}[Block decomposition]\label{blockdecom}
    If \(m \in \mathcal{M}(\Omega) \cap BV(\Omega)\) is not block decomposable, then it is not an optimal control of (\ref{opt}). Furthermore, if \(m\) is additionally bang-bang type, we can find \(\hat{m} \in BV(\Omega)\) such that \(0 \leq \hat{m} \leq 1\),
    \begin{equation}\label{con1}
        \int_\Omega \theta_{m, \mu} = \int_\Omega \theta_{\hat{m}, \mu},
    \end{equation}
    and
    \begin{equation}\label{con2}
        \int_\Omega m > \int_\Omega \hat{m}.
    \end{equation}
\end{theorem}

Each block solution of \(\theta_{m, \mu}\) on subdomains can be characterized as \(\theta_{l, b}\), which is a weak solution of
\begin{equation}
    \begin{cases}
        \theta_{l, b}'' + \theta_{l, b}(\chi_{(0, b)} - \theta_{l, b}) = 0 & \text{in } (0, l),
        \\
        \theta_{l, b}' = 0 & \text{on } 0, l.
    \end{cases}
\end{equation}
obtained by translation and rescaling.

We can calculate how much advantage \(\theta_{l, b}\) gets compared to the resource \(\chi_{(0, l)}\) in the subdomain \((0, l)\). We will define it as the advantage function \(H\) of \(l\) and \(b\).
\begin{equation}
    H(l, b) = \int_{(0, l)} \theta_{l, b} - b.
\end{equation}

For given \(m\) and \(\mu\), let \(\theta_{l_1, b_1}, \theta_{l_2, b_2}, \cdots, \theta_{l_r, b_r}\) be the block solutions of \(\theta_{m, \mu}\). Then we have
\begin{equation}
    b_1 + b_2 + \cdots + b_r = \frac{m_0}{\sqrt{\mu}},
\end{equation}
\begin{equation}
    l_1 + l_2 + \cdots + l_r = \frac{1}{\sqrt{\mu}},
\end{equation}
and
\begin{equation}
    \sum_i H(l_i, b_i) = \frac{F_\mu(\theta_{m, \mu}) - m_0}{\sqrt{\mu}}.
\end{equation}

\begin{figure}[htb!]
    \centerline{\includegraphics[scale = 0.65]{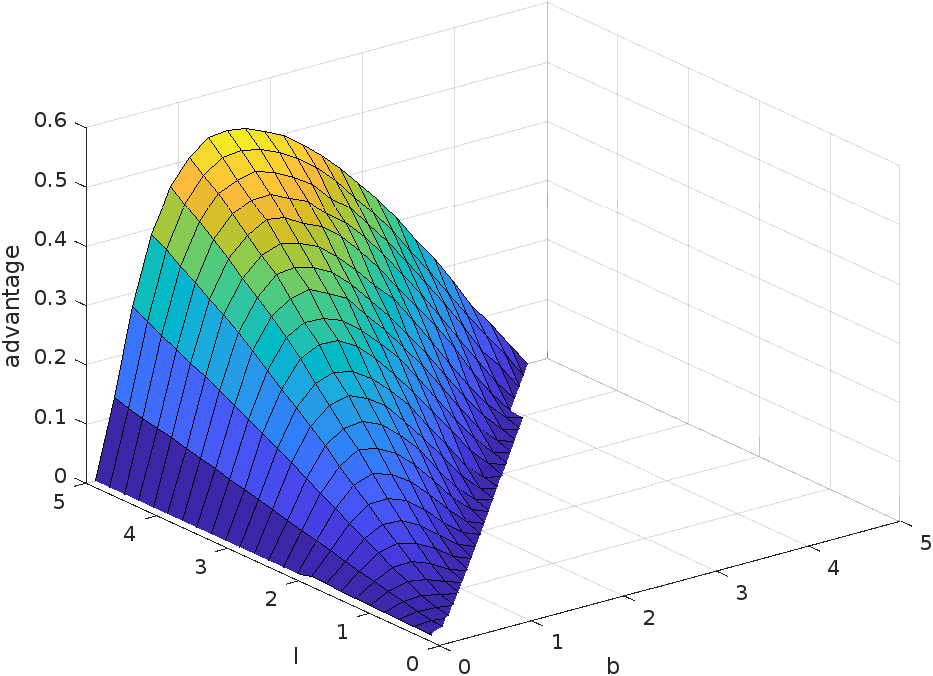}}
    \caption{Advantage Function \(H(l, b)\)}
    \label{2}
\end{figure}

Hence, it is important to investigate the convexity of \(H\) to find an optimal control. As shown in Figure \ref{2}, we can roughly observe that \(H\) is convex upwards on its domain except when \(b/l \ll 1\), \(1 - b/l \ll 1\) or \(l \ll 1\). In other words, the advantage function is convex upwards except the area near the boundary of the domain. Such an observation shows the importance of the analysis of \(H\) when \(b/l \ll 1\), \(1 - b/l \ll 1\) or \(l \ll 1\). In this context, we showed that for any constant \(0 < C_1 < C_2\), there exist sufficiently small \(\alpha\) such that
\begin{equation}\label{ineqad}
    H(l_1, b_1) + H(l_2, b_2) \leq H(l_1 + l_2, b_1 + b_2),
\end{equation}
for any \(l_i\) and \(b_i\) satisfying \(C_1 \alpha < b_i/l_i < C_2\alpha\) for \(i = 1, 2\).

In addition, (\ref{ineqad}) gives an additional result as follows.

\begin{theorem}[Optimal control with small resource]\label{optimal}
    There exists \(m_1(|\Omega|, \mu) \in (0, 1)\) such that \(m_\mu^*(x) = \chi_{(0, b)}(x)\) or \(m_\mu^*(x) = \chi_{(0, b)}(|\Omega| - x)\) if \(m_0 < m_1\) where \(b = m_0|\Omega|\).
\end{theorem}

\section{Block Decomposition}

In this section, we investigate the local properties of \(\theta_{m, \mu}\) and \(m \in \mathcal{M}(\Omega) \cap BV(\Omega)\) and prove Theorem \ref{blockdecom}. In particular, we are interested in scale-invariant properties of \(\theta_{m, \mu}\). Let \(\tilde{\theta}(x) := \theta_{m, \mu}(x\sqrt{\mu})\), \(\tilde{m}(x) := m(x\sqrt{\mu})\), and \(\tilde{\Omega} = \frac{1}{\sqrt{\mu}}\Omega\). Then \(\tilde{\theta}\) is a weak solution of
\begin{equation}\label{rescale}
    \begin{cases}
        \tilde{\theta}'' + \tilde{\theta}(\tilde{m} - \tilde{\theta}) = 0 & \text{in } \tilde{\Omega}
        \\
        \dfrac{\partial\tilde{\theta}}{\partial\nu} = 0 & \text{on } \partial\tilde{\Omega}.
    \end{cases}
\end{equation}
Then we can observe that \(\tilde{\theta}\) and \(\tilde{m}\) share the same scale-invariant properties. This fact is useful in proving the following lemmas in this section.

The first property we obtain is the property of critical points of \(\theta_{m, \mu}\).

\begin{lemma}\label{finite}
    For any bang-bang type \(m \in \mathcal{M}(\Omega) \cap BV(\Omega)\), \(0 < \theta_{m, \mu}(x) < 1\) holds for all \(x \in \bar{\Omega}\). Moreover, there are finitely many local maximum and minimum points of \(\theta_{m, \mu}\)
\end{lemma}

\begin{proof}
    By rescaling if necessary, we may let \(\mu = 1\). Moreover, let us denote \(\theta := \theta_{m, \mu}\) for convenience. Let us write
    \begin{equation}
        m = \sum_{i = 1}^n \chi_{I_i} \text{ a.e.}
    \end{equation}
    where \(I_i = [a_i, b_i] \subset \bar \Omega = [a, b]\), and \(b_i < a_{i + 1}\). Moreover, set \(b_0 = a, a_{n + 1} = b\) for convenience.

    Now, assume that there exists \(x_0 \in \bar{\Omega}\) such that \(\theta(x_0) = 0\). If there exists \(I_i\) such that \(x_0 \in I_i\), then \(\theta \equiv 0\) on \(I_i\) by the maximum principle and the Hopf lemma. Thus, \(\theta(a_i) = \theta(b_i) = 0\), and we may assume that there exists \(x_0 \in [b_i, a_{i + 1}]\) such that \(\theta(x_0) = 0\).

    Let us further suppose that we cannot find a pair of points \(x_0, x_1 \in [b_i, a_{i + 1}]\) such that \(\theta(x_0) = 0\) and \(\theta(x_1) > 0\). In other words, if \(\theta(x_0) = 0\) for some \(x_0 \in [b_i, a_{i + 1}]\), \(\theta(x) \equiv 0\) on \(x \in [b_i, a_{i + 1}]\). However, if \(\theta(x) \equiv 0\) on \(x \in [b_i, a_{i + 1}]\), \(\theta(x) = 0\) at the boundaries of both \(I_i\) and \(I_{i + 1}\). Then we obtain a contradiction by applying the Hopf lemma unless \(\theta \equiv 0\) on \(x \in I_i, I_{i + 1}\) which implies that \(\theta \equiv 0\) on \(I_i\) and \(I_{i + 1}\). By iterating this argument, it further implies that \(\theta \equiv 0\) on \(\Omega\).
    
    Thus, there exists \(x_1 \in [b_i, a_{i + 1}]\) such that \(\theta_{m, \mu}(x_1) > 0\). Then we get
    \begin{equation}\label{in}
        |\Omega| \geq |x_1 - x_0| = \left|\int_{x_0}^{x_1} dx\right| \geq \left|\int_{0}^{\theta'(x_1)} \frac{1}{\theta''}d\theta'\right| = \left|\int_{0}^{\theta'(x_1)} \frac{1}{\theta^2}d\theta'\right|
    \end{equation}
    The equation of \(\theta\) and \(\theta'\) for \(x\) between \(x_0\) and \(x_1\) is
    \begin{equation}\label{put}
        (\theta')^2 = \frac{2}{3}\theta^3.
    \end{equation}
    Plugging (\ref{put}) into (\ref{in}), we get
    \begin{equation}
        |\Omega| \geq \int_0^{|\theta'(x_1)|} \frac{1}{(\frac{3}{2})^\frac{2}{3}t^\frac{4}{3}}dt
    \end{equation}
    Note that if \(\theta'(x_1) = 0\), then \(\theta = 0\) between \(x_0\) and \(x_1\). This contradicts that \(\theta(x_1)> 0\). Thus, this integration is improper and diverges to infinity, which is a contradiction. Therefore, there is no \(x_0 \in \bar{\Omega}\) such that \(\theta(x_0) = 0\).
    
    In a similar way, If there exists a point in \(\bar \Omega\) where \(\theta\) attains 1, we can find a pair of points \(x_3, x_4 \in I_i\) such that \(\theta(x_3) = 1\) and \(\theta(x_4) < 1\).
    Then we get
    \begin{equation}\label{in2}
        |\Omega| \geq |x_4 - x_3| = \left|\int_{\theta'(x_3)}^{\theta'(x_4)} \frac{1}{\theta''}d\theta'\right| = \left|\int_{0}^{\theta'(x_4)} \frac{1}{\theta(1 - \theta)}d\theta'\right|
    \end{equation}
    The equation of \(\theta\) and \(\theta'\) for \(x\) between \(x_3\) and \(x_4\) is
    \begin{equation}\label{put2}
        (\theta')^2 = \frac{2}{3}\theta^3 - \theta^2 + \frac{1}{3} = \frac{1}{3}(\theta-1)^2(2\theta + 1).
    \end{equation}
    Plugging (\ref{put2}) into (\ref{in2}), we get
    \begin{equation}
        |\Omega| \geq \left|\int_0^{\theta'(x_4)} \frac{\sqrt{2\theta + 1}}{\sqrt{3}\theta\theta'}d\theta'\right| \geq \int_0^{|\theta'(x_4)|} \frac{1}{\sqrt{3}t}dt
    \end{equation}
    Note that if \(\theta'(x_4) = 0\), then \(\theta = 1\) between \(x_3\) and \(x_4\). This contradicts that \(\theta(x_4) < 1\). Thus, this integration is improper and diverges to infinity, which is a contradiction. Therefore, there is no point in \(\bar{\Omega}\) where \(\theta\) attains 1, and \(0 < \theta < 1\) holds on \(\bar{\Omega}\).

    Now, let us prove that there are only finitely many local maximum and minimum points. It is enough to prove that there exists at most one critical point in \((a_i, b_i)\) or \((b_i, a_{i + 1})\). Since the sign of \(\theta''\) does not change in \((a_i, b_i)\) or \((b_i, a_{i + 1})\), if there exists more than two critical points in these intervals, there exists a constant \(c \in (0, 1)\) such that \(\theta \equiv c\) between these critical points. Then we get \(m \equiv \theta \equiv c\) for some interval, which is a contradiction because \(m\) is bang-bang.
\end{proof}

By the lemma above, we may denote all local extrema of \(\theta_{m, \mu}\) by \(x_1 < x_2 < \cdots < x_r\). Moreover, by setting \(x_0 = a, x_{r + 1} = b\), we observe that \(\theta_{m, \mu}\) is monotone in \((x_i, x_{i + 1})\) for \(0 \leq i \leq r\).

We claim that if \(m\) is not monotone in \((x_i, x_{i+1})\), then \(m\) is not an optimal control of (\ref{opt}). In other words, if \(m\) is not in the form of
\begin{equation}
    m(x) = \chi_{(x', x_{i + 1})}(x) \text{ nor } m(x) = \chi_{(x_i, x')}(x) \text{ a.e}
\end{equation}
in \((x_i, x_{i + 1})\) for some \(x' \in (x_i, x_{i + 1})\), it is not an optimal control.

\begin{lemma}\label{mono}
    Suppose that \(\theta_{m, \mu}\) is monotone in the sub-interval \(I \subset \Omega\) while \(m\) is not monotone in \(I\). Then \(m\) is not an optimal control of (\ref{opt}).
\end{lemma}

\begin{proof}
    By rescaling if it is necessary, we may let \(\mu = 1\). Also, let us denote \(\theta := \theta_{m, \mu}\) for convenience. Moreover, we can find \(0 \leq i \leq r\) such that \(I \subset (x_i, x_{i + 1})\) where \(x_i\)s are defined as above. Without loss of generality, let us assume that \(\theta\) is increasing in \((x_i, x_{i + 1})\). We observe that there exist \(x_i \leq c_0 < c_1 < c_2 < c_3 < c_4 \leq x_{i + 1}\) such that
    \begin{equation}\label{nmono}
        m(x) = \chi_{(c_1, c_2)}(x) + \chi_{(c_3, c_4)}(x) \text{ a.e}
    \end{equation}
    on \(x \in (c_0, c_4)\). If no such \(c_0, \cdots, c_4\) exist, \(m\) must be in the form of
    \begin{equation}
        m(x) = \chi_{(x_i, d_1)}(x) + \chi_{(d_2, d_3)}(x) \text{ or } m(x) = \chi_{(d_1, d_2)}(x) \text{ a.e}
    \end{equation}
    for some \(x_i \leq d_1 < d_2 < d_3 \leq x_{i + 1}\). However, if \(m(x) = \chi_{(x_i, d_1)}(x) + \chi_{(d_2, d_3)}(x)\), \(\theta\) must decrease in \((x_i, d_1)\) since \(\theta'(x_i) = 0\), and \(\theta'' < 0\) on \((x_i, d_1)\). On the other hand, if \(m(x) = \chi_{(d_1, d_2)}(x)\), \(\theta\) must decrease in \((d_2, x_{i + 1})\) since \(\theta'(x_{i + 1}) = 0\), and \(\theta'' > 0\) on \((d_2, x_{i + 1})\). In either case, a contradiction is induced.
    
    Now, we will draw a phase diagram of \(\theta\) for \(x \in (c_0, c_4)\), and construct a new function in the diagram. It is depicted in the following Figure \ref{3}.
    
\begin{figure}[htb!]
    \centering
    \begin{tikzpicture}[scale = 3.3]
        \draw[thick, ->](-0.1, 0)--(1.1, 0) node [anchor=north]{\(\theta\)};
        \draw[thick, ->](0, -0.1)--(0, 0.9) node [anchor=east]{\(\theta'\)};
        \node[below] at (-0.2, 0) {\(0\)};
        \draw[domain=0.05:0.3, smooth,thick,,variable=\x,black] plot ({\x},{sqrt(1 - (\x - 1)*(\x - 1)) - 0.2});
        \draw[domain=0.3:0.5, smooth,thick,,variable=\x,blue] plot ({\x},{sqrt(1 - (\x - 1)*(\x - 1)) - 0.2});
        \draw[domain=0.5:0.7, smooth,thick,,variable=\x,blue] plot ({\x},{sqrt(1 - \x*\x) - 0.2});
        \draw[domain=0.7:0.95, smooth,thick,,variable=\x,black] plot ({\x},{sqrt(1 - \x*\x) - 0.2});
        \draw[domain=0.3:0.5, smooth,thick,,variable=\x,black] plot ({\x},{sqrt(1 - (\x + 0.4)*(\x + 0.4)) - 0.2});
        \draw[domain=0.5:0.7, smooth,thick,,variable=\x,black] plot ({\x},{sqrt(1 - (\x - 1.4)*(\x - 1.4)) - 0.2});
        \draw[domain=0.4:0.6, smooth,thick,,variable=\x,red] plot ({\x},{sqrt(1 - (\x - 1.2)*(\x - 1.2)) - 0.2});
        \node[left] at (0.3, 0.55) {\(A\)};
        \node[below] at (0.5, 0.25) {\(B\)};
        \node[above] at (0.5, 0.65) {\(D\)};
        \node[right] at (0.7, 0.55) {\(C\)};
        \node[left] at (0.41, 0.36) {\(E\)};
        \node[right] at (0.57, 0.65) {\(F\)};
    \end{tikzpicture}
    \caption{The Phase Diagram of \(\theta\)}
    \label{3}
\end{figure}
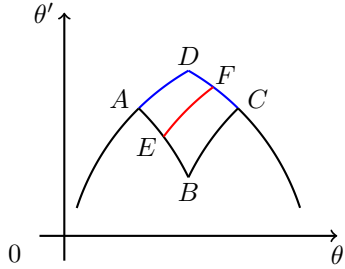

The black line is the phase diagram for the original solution \(\theta\), and we consider a change in the diagram from \(A \to B \to C\) to \(A \to D \to C\) which is depicted as the blue line.

Note that \(A\) corresponds to the point \(x = c_1\), \(B\) corresponds to the point \(x = c_2\), and \(C\) corresponds to the point \(x = c_3\). We may define the curve \((\theta_1, \theta_1')\) which continues \(\{(\theta(x), \theta'(x)): x \in (c_0, c_1)\}\) by
\begin{equation}\label{1c}
    \left\{(\theta_1(x), \theta_1'(x)):(\theta_1')^2 - (\theta'(c_0))^2 = \frac{2}{3}\theta_1^3 - \frac{2}{3}\theta^3(c_0)\right\}
\end{equation}
Hence, we understand \(\theta_1\) as an extension of \(\theta|_{(c_0, c_1)}\) corresponding to the resource \(m \equiv 0\), since \(m = 0\) in \((c_0, c_1)\).

Moreover, we define the curve \((\theta_2, \theta_2')\) which continues \(\{(\theta(x), \theta'(x)): x \in (c_1, c_2)\}\) by,
\begin{equation}\label{2c}
    \left\{(\theta_2(x), \theta_2'(x)):(\theta_2')^2 - (\theta'(c_1))^2 = \frac{2}{3}\theta_2^3 - \theta_2^2 - \left(\frac{2}{3}\theta^3(c_1) - \theta^2(c_1)\right)\right\},
\end{equation}
define the curve \((\theta_3, \theta_3')\) which continues \(\{(\theta(x), \theta'(x)): x \in (c_2, c_3)\}\) by,
\begin{equation}\label{3c}
    \left\{(\theta_3(x), \theta_3'(x)):(\theta_3')^2 - (\theta'(c_2))^2 = \frac{2}{3}\theta_3^3 - \frac{2}{3}\theta^3(c_2)\right\},
\end{equation}
and define the curve \((\theta_4, \theta_4')\) which continues \(\{(\theta(x), \theta'(x)): x \in (c_3, c_4)\}\) by,
\begin{equation}\label{4c}
    \left\{(\theta_4(x), \theta_4'(x)):(\theta_4')^2 - (\theta'(c_3))^2 = \frac{2}{3}\theta_4^3 - \theta_4^2 - \left(\frac{2}{3}\theta^3(c_3) - \theta^2(c_3)\right)\right\}.
\end{equation}
We understand \(\theta_2, \theta_3\) and \(\theta_4\) in a similar manner as we understand \(\theta_1\).

Finally, we will set \(D = (\theta_1(d_-), \theta_1'(d_-)) = (\theta_4(d_+), \theta_4'(d_+))\) as the intersection of (\ref{1c}) and (\ref{4c}) in \(\{\theta' > 0\}\).

Now, we will define a function \(\hat{\theta}_{m, \mu}\) as
\begin{equation}
    \hat{\theta} = 
    \begin{cases}
        \theta(x) & \text{in } \Omega \setminus (x_i, x_{i + 1})
        \\
        \theta(x_i) & \text{in } (x_i, \hat{x}_1]
        \\
        \theta_1(x - \hat{x}_1 + x_i) & \text{in } [\hat{x}_1, \hat{x}_2]
        \\
        \theta_4(x - \hat{x}_3 + x_{i + 1}) & \text{in } [\hat{x}_2, \hat{x}_3]
        \\
        \theta(x_{i + 1}) & \text{in } [\hat{x}_3, x_{i + 1})
    \end{cases}
\end{equation}
for some fixed \(x_i < \hat{x}_1 < \hat{x}_2 < \hat{x}_3 < x_{i+1}\) defined below.

\begin{equation}
    \begin{split}
        \hat{x}_2 = & \hat{x}_1 + d_- - x_i,
        \\
        \hat{x}_3 = & \hat{x}_2 + x_{i + 1} - d_+.
    \end{split}
\end{equation}
To show that \(\hat{\theta}\) is well-defined, it suffices to show that \(\hat{x}_3 = \hat{x}_1 + d_- - d_+ + x_{i+1} - x_i < x_{i+1}\). In other words, it is enough to show that
\begin{equation}\label{gap}
    d_- < d_+.
\end{equation}

To prove this, let us introduce a new curve \((\theta_5, \theta_5')\) that depends on \(t\) by
\begin{equation}\label{5c}
    \left\{(\theta_5, \theta_5'):(\theta_5')^2 - (\theta'(c_2))^2 = \frac{2}{3}\theta_5^3 - \frac{2}{3}\theta^3(c_2) + t\right\},
\end{equation}
which is depicted in Figure \ref{3} as the red curve. Also, denote \(E(t)\) as the intersection of (\ref{2c}) and (\ref{5c}) and \(F(t)\) as the intersection of (\ref{4c}) and (\ref{5c}) in \(\{\theta' > 0\}\) for \(t \geq 0\) such that
\begin{equation}
    0 \leq t \leq T := - \frac{3}{2}(\theta^3(c_0) - \theta^3(c_2)) + (\theta'(c_0))^2 - (\theta'(c_2))^2.
\end{equation}
In addition, let us define \(e_-(t), e_+(t), f_-(t)\) and \(f_+(t): [0, T] \to \mathbb{R}\) satisfy \(E = (\theta_5(e_-), \theta'_5(e_-)) = (\theta_2(e_+), \theta'_2(e_+))\) and \(F = (\theta_5(f_-), \theta'_5(f_-)) = (\theta_3(f_+), \theta'_3(f_+))\).

Note that \(\theta (E) = \sqrt{\theta^2(B) - t}\), and \(\theta (F) = \sqrt{\theta^2(C) - t}\). Moreover, we observe that \(E(0) = B, E(T) = A, F(0) = C\) and \(F(T) = D\). Then, we may introduce a function \(\zeta(t) : [0, T] \to \mathbb{R}\) as
\begin{equation}\label{defzeta}
    \zeta(t) = f_- - e_- - f_+ + e_+.
\end{equation}
We have \(\zeta(0) = 0\) and \(\zeta(T) = d_ - - d_+\). Moreover, it is easy to verify that \(\zeta\) is continuous on \([0, T]\), and differentiable on \((0, T)\) by the Leibniz integral rule. In detail, we get the derivative of \(\zeta\) as below.
\begin{equation}\label{difzeta}
    \begin{split}
        \zeta'(t) = & \frac{d}{dt}\left((f_+ - e_+) + (c_2 - e_-) - (c_3 - f_-)\right)
        \\ = & \frac{d}{dt}\left( \int_{\theta(E)}^{\theta(F)} \frac{1}{\theta_5'} d\theta_5 + \int_{\theta(F)}^{\theta(C)} \frac{1}{\theta_4'} d\theta_4 - \int_{\theta(E)}^{\theta(B)} \frac{1}{\theta_2'} d\theta_2 \right)
        \\ = & \int_{\theta(E)}^{\theta(F)} \frac{d}{dt}\frac{1}{\theta_5'}d\theta_5 = - \int_{\theta(E)}^{\theta(F)} \frac{1}{2(\theta_5')^3}d\theta_5 < 0
    \end{split}
\end{equation}
Thus, \(\zeta(T) = \int_0^T \zeta'(t)dt = d_- - d_+ < 0\), which implies that \(\hat{\theta}\) is well-defined.

Next, we will find \(\hat{x}_1\) such that \(\int_\Omega \hat\theta = \int_\Omega \theta\). By construction, we have
\begin{equation}
    \begin{split}
        \int_\Omega \hat{\theta} - \int_\Omega \theta = & \theta(x_i)(\hat{x}_1 - x_i) + \theta(x_{i + 1})(x_{i + 1} - \hat{x}_3)
        \\
        & - \int_{c_1}^{c_2} \theta_2 dx - \int_{c_2}^{c_3} \theta_3 dx + \int_{c_2}^{d_-} \theta_1 dx + \int_{d_+}^{c_3} \theta_4 dx
    \end{split}
\end{equation}
Note that, by (\ref{gap}), we have \(x_1 \leq \hat{x}_1 \leq x_1 - \zeta(T)\) and
\begin{equation}
    \theta(x_i)(\hat{x}_1 - x_i) + \theta(x_{i + 1})(x_{i + 1} - \hat{x}_4) = (\theta(x_i) - \theta(x_{i + 1}))(\hat{x}_1 - x_1) - \theta(x_{i + 1})\zeta(T).
\end{equation}
Hence, it is enough to prove that
\begin{equation}
    \theta(x_{i + 1})\zeta(T) \leq - \int_{c_1}^{c_2} \theta_2 dx - \int_{c_2}^{c_3} \theta_3 dx + \int_{c_2}^{d_-} \theta_1 dx + \int_{d_+}^{c_3} \theta_4 dx \leq \theta(x_i)\zeta(T)
\end{equation}
to show the existence of \(\hat{x}_1\). If we introduce a function \(\xi(t):[0, T] \to \mathbb{R}\) as
\begin{equation}
    \xi(t) = - \int_{e_-}^{c_2} \theta_2 dx - \int_{c_2}^{c_3} \theta_3 dx + \int_{e_+}^{f_+} \theta_5 dx + \int_{f_-}^{c_3} \theta_4 dx,
\end{equation}
then we have \(\xi(0) = 0\) and \(\xi(T) = - \int_{c_1}^{c_2} \theta_2 dx - \int_{c_2}^{c_3} \theta_3 dx + \int_{c_2}^{d_-} \theta_1 dx + \int_{d_+}^{c_3} \theta_4 dx\). Moreover, it is easy to verify that \(\xi\) is continuous on \([0, T]\), and differentiable on \((0, T)\) by the Leibniz integral rule. In detail, we get the derivative of \(\xi\) as follows.
\begin{equation}
    \xi'(t) =\int_{\theta(E)}^{\theta(F)} -\frac{\theta_5}{2(\theta_5')^3}d\theta_5.
\end{equation}
Since \(\theta(x_i) \leq \theta(x) \leq \theta(x_{i + 1})\) holds on \(x \in [x_i, x_{i + 1}]\), we get
\begin{equation}
    \theta(x_{i + 1})\zeta'(t) \leq \xi'(t) \leq \theta(x_i)\zeta'(t).
\end{equation}
By taking the integral on \([0, T]\), we get \(\theta(x_{i + 1})\zeta(T) \leq \xi(T) \leq \theta(x_i)\zeta(T)\), which is the desired result. To be specific, take \(\hat{x}_1\) as
\begin{equation}
    \hat{x}_1 = x_i + \frac{\xi(T) - \theta(x_{i + 1})\zeta(T)}{\theta(x_{i + 1}) - \theta(x_i)},
\end{equation}
then we have \(\int_\Omega \hat{\theta} = \int_\Omega \theta\).

Now, we observe that \(\hat{\theta}\) is a weak solution of
\begin{equation}
    \begin{cases}
        \hat{\theta}'' + \hat{\theta}(\hat{m} - \hat{\theta}) = 0 & \text{in } \Omega
        \\
        \hat{\theta}' = 0 & \text{on } \partial\Omega,
    \end{cases}
\end{equation}
where \(\hat{m} \in \mathcal{M}(\Omega)\) is defined by
\begin{equation}\label{mhat}
    \hat{m} = 
    \begin{cases}
        m(x) & \text{in } \Omega \setminus (x_i, x_{i + 1})
        \\
        \theta(x_i) & \text{in } (x_i, \hat{x}_1)
        \\
        0 & \text{in } [\hat{x}_1, \hat{x}_2]
        \\
        1 & \text{in } [\hat{x}_2, \hat{x}_3]
        \\
        \theta(x_{i + 1}) & \text{in } (\hat{x}_3, x_{i + 1})
    \end{cases}
\end{equation}

Since \(\int_\Omega \hat \theta = \int_\Omega \theta\), to show that \(m\) is not an optimal control it suffices to show that
\begin{equation}
    \int_\Omega \hat{m} - \int_\Omega m = - \xi(T) - c_2 + c_1 + c_3 - d_+ < 0.
\end{equation}
To do so, let us introduce a function \(\eta(t) : [0, T] \to \mathbb{R}\) as
\begin{equation}\label{defeta}
    \eta(t) = - c_2 + e_+ + c_3 - f_+.
\end{equation}
Then, we have \(\eta(0) = 0\) and \(\eta(T) = - c_2 + c_1 + c_3 - d_+\). Moreover, it is easy to verify that \(\eta\) is continuous on \([0, T]\), and differentiable on \((0, T)\) by the Leibniz integral rule. In detail, we get the derivative of \(\eta\) as below.

\begin{equation}\label{difeta}
    \begin{split}
        \eta'(t) = & \frac{d}{dt}\left(- c_2 + e_+ + c_3 - f_+\right)
        \\ = & -\frac{1}{\theta'(E)}\cdot\frac{d\theta(E)}{dt} + \frac{1}{\theta'(F)}\cdot\frac{d\theta(F)}{dt}
        \\ = & \left[\frac{1}{2\theta\theta'}\right]^F_E
        \\ = & \int_{\theta(E)}^{\theta(F)} \frac{d}{d\theta_5}\left(\frac{1}{2\theta_5\theta'_5}\right)d\theta_5
        \\ = & \int_{\theta(E)}^{\theta(F)} -\frac{1}{2(\theta_5')^3}\left(\frac{(\theta_5')^2}{\theta_5^2} + \theta_5\right)d\theta_5 = \xi'(t) - \int_{\theta(E)}^{\theta(F)} \frac{1}{2\theta_5^2\theta_5'}d\theta_5.
    \end{split}
\end{equation}
Thus, we get \(\eta'(t) - \xi'(t) < 0\). By taking the integral on \([0, T]\), we get \(\eta(T) - \xi(T) < 0\). Therefore,
\begin{equation}
    \int_\Omega \hat{m} - \int_\Omega m = -\xi(T) + \eta(T) < 0,
\end{equation}
which implies that \(m\) is not an optimal control of (\ref{opt}).
\end{proof}

Then, Theorem \ref{blockdecom} follows from Lemma \ref{mono}.

\begin{proof}[Proof of Theorem \ref{blockdecom}]
    Suppose that \(m \in \mathcal{M}(\Omega) \cap BV(\Omega)\) is not block decomposable. If \(m\) is not a bang-bang type resource, it is not an optimal control by Theorem \ref{bang}. Hence, we may assume that \(m\) is bang-bang.

    Next, by Lemma \ref{finite}, we can take all local maximum and local minimum points of \(\theta_{m, \mu}\) as \(a = x_0 < x_1 < \cdots < x_{r + 1} = b\) where \(\Omega = (a, b)\). Then, by Lemma \ref{mono}, \(m\) must be monotone on each sub-intervals \((x_i, x_{i + 1})\) for \(i = 0, 1, \cdots, r\). Hence, (\ref{monom}) holds for \(i = 0, 1, \cdots, r\), which implies that \(m\) is block decomposable. Therefore, a contradiction is deduced.

    On the other hand, if \(m\) is not block decomposable and is bang-bang, we can find \(0 \leq i \leq r\) and \(x_i \leq c_0 < c_1 < c_2 < c_3 < c_4 \leq x_{i + 1}\) which satisfies (\ref{nmono}). Then, we can construct \(\hat m\) as (\ref{mhat}), which satisfies the desired results (\ref{con1}) and (\ref{con2}).
\end{proof}

To investigate general bang-bang type resource \(m \in \mathcal{M}(\Omega) \cap BV(\Omega)\) in the next section, we need further discussion about slight modification \(\hat m\) of \(m\) defined as (\ref{mhat}). Observe that
\begin{equation}
    \int_\Omega \chi_{\{\hat m = 1\}} - \int_\Omega \chi_{\{ m = 1\}} = \eta(T) < 0
\end{equation}
where \(\hat m\) is defined as (\ref{mhat}), and \(\eta\) is defined as (\ref{defeta}). Moreover,
\begin{equation}
    \int_\Omega \chi_{\{\hat m = 0\}} - \int_\Omega \chi_{\{ m = 0\}} = \zeta(T) - \eta(T)
\end{equation}
where \(\zeta\) is defined as (\ref{defzeta}).

By (\ref{difzeta}) and (\ref{difeta}), we have
\begin{equation}
    \zeta'(t) - \eta'(t) = -\int_{\theta(E)}^{\theta(F)} \frac{1}{2(\theta'_5)^3}\left(1 - \frac{(\theta_5')^2}{\theta_5^2} - \theta_5\right)d\theta_5
\end{equation}
for \(0 \leq t \leq T\) where \(\theta_5\) follows (\ref{5c}). Observe that
\begin{equation}
    (\theta_5')^2 \leq \frac{2}{3}\theta_5^3
\end{equation}
holds, so
\begin{equation}
    1 - \frac{(\theta_5')^2}{\theta_5^2} - \theta_5 \geq 1 - \frac{5}{3}\theta_5.
\end{equation}

Thus, \(\zeta'(t) - \eta'(t) \leq 0\) holds if \(\theta_5(F) \leq \theta(c_3) \leq \frac{3}{5}\). In addition, recall the following lemma.
    \begin{lemma}\label{bound}
        For any \(\mu_0 > 0\), there exists \(m_1 > 0\) such that
        \begin{equation}
            C_1m_0 \leq \theta_{m, \mu} \leq C_2m_0.
        \end{equation}
        for some \(C_1(\mu_0), C_2(\mu_0) > 0\) for any \((m_0, \mu) \in (0, m_1)\times(\mu_0, \infty)\).
        
        In other words, \(m_0 \sim \theta_{m, \mu}\) for small enough \(m_0 > 0\).
    \end{lemma}

\begin{proof}
    By rescaling if necessary, we may let \(\mu = 1\). Also, let us denote \(\theta := \theta_{m, \mu}\) for convenience.
    
    Let us denote \(\theta_{min} := \inf_{\bar\Omega} \theta = \theta(x_1)\) and \(\theta_{max} := \sup_{\bar\Omega} \theta = \theta(x_2)\) for some \(x_1, x_2 \in \bar\Omega\). Then we have
    \begin{equation}
        (\theta')^2 \leq \frac{2}{3}(\theta^3 - \theta^3_{min}).
    \end{equation}

    Now, let \(\Theta \in C^2(\bar\Omega)\) be a solution of following initial value problem.
    \begin{equation}
        \begin{cases}
            \Theta '' = \Theta^2 & \textnormal{on } \Omega,
            \\
            \Theta(x_1) = \theta_{min}, \Theta_1'(x_1) = 0.
        \end{cases}
    \end{equation}
    Then, \(\theta(x) \leq \Theta(x)\) holds for any \(x \in \bar\Omega\) since we have
    \begin{equation}
        \Theta'(x) =
        \begin{cases}
            \sqrt\frac{2}{3}(\Theta^3(x) - \Theta^3_{min})^\frac{1}{2} & \textnormal{if } x \geq x_1,
            \\
            -\sqrt\frac{2}{3}(\Theta^3(x) - \Theta^3_{min})^\frac{1}{2} & \textnormal{if } x \leq x_1.
        \end{cases}
    \end{equation}
    In addition, if \(\theta_{min}\) is sufficiently small, by applying the Taylor expansion of \(\Theta\) at \(x_1\), we get \(\Theta = \theta_{min} + O(\theta_{min}^2)\). In other words, there exists a constant \(c > 0\) such that
    \begin{equation}
        \Theta \leq c \theta_{min}.
    \end{equation}
    Therefore, we get
    \begin{equation}
        \theta_{max} \leq \Theta(x_2)\leq c \theta_{min}.
    \end{equation}

    Recall that, by Theorem \ref{boundineq}, we have
    \begin{equation}
        \theta_{min} \leq 3m_0
    \end{equation}
    and
    \begin{equation}
        \theta_{max} \geq m_0.
    \end{equation}
    Thus, we get
    \begin{equation}
        \frac{m_0}{c} \leq \theta_{min} \leq \theta \leq \theta_{max} \leq 3cm_0
    \end{equation}
    as desired.
\end{proof}

By the lemma above, we have \(\theta(c_2) \leq \frac{3}{5}\) if \(m_0\) is sufficiently small. Therefore,
\begin{equation}
    \int_\Omega \chi_{\{\hat m = 0\}} - \int_\Omega \chi_{\{ m = 0\}} = \zeta(T) - \eta(T) < 0
\end{equation}
holds if \(m_0\) is sufficiently small, and we get the following corollary.

\begin{corollary}\label{modif}
    For any \(\mu_0 > 0\), there exists \(m_1 > 0\) such that if \((\mu, m_0) \in (\mu_0, \infty) \times (0, m_1)\), and a bang-bang type resource \(m \in \mathcal{M}(\Omega) \cap BV(\Omega)\) is not block decomposable, there exists \(\hat{m} \in BV(\Omega)\) that satisfies \(0 \leq \hat{m} \leq 1\), \textnormal{(\ref{con1})}, \textnormal{(\ref{con2})},
    \begin{equation}
        \int_\Omega \chi_{\{\hat m = 1\}} - \int_\Omega \chi_{\{ m = 1\}} < 0,
    \end{equation}
    and
    \begin{equation}
        \int_\Omega \chi_{\{\hat m = 0\}} - \int_\Omega \chi_{\{ m = 0\}} < 0.
    \end{equation}
\end{corollary}

Furthermore, by applying this slight modification repeatedly on a bang-bang type block indecomposable \(m\), we get refined resource \(\tilde m\) such that there exists a partition \(a = y_0 < y_1 < \cdots < y_{s+1} = b \) of \(\Omega = (a, b)\) which satisfies \(\{y_0, y_1, \cdots, y_{s + 1}\} \supset \{x_0, x_1, \cdots, x_{r + 1}\}\) where \(\{x_0, x_1, \cdots, x_{r + 1}\}\) are the set of all critical points of \(\theta_{m, \mu}\), \(\theta_{\tilde m, \mu}' (y_i) = 0\) for \(i = 0, 1, \cdots, s+1\), and either of the following must hold for each \(i = 0, 1, \cdots, s\).
\begin{enumerate}
    \item [(i)] There exists \(y'_i \in (y_i, y_{i+1})\) such that
    \begin{equation}
        \tilde m(x) = \chi_{(y'_i, y_{i + 1})}(x) \text{ or } \tilde m(x) = \chi_{(y_i, y'_i)}(x) \text{ a.e.}
    \end{equation}
    \item [(ii)] \(\theta_{\tilde m, \mu}(x)\) is constant on \((y_i, y_{i+1})\).
\end{enumerate}

We say that \(\tilde m\) is a block-refined resource of \(m\). Then it satisfies the following corollary.

\begin{corollary}\label{blockrefined}
    For any \(\mu_0 > 0\), there exists \(m_1 > 0\) such that if \((\mu, m_0) \in (\mu_0, \infty) \times (0, m_1)\), and a bang-bang type resource \(m \in \mathcal{M}(\Omega) \cap BV(\Omega)\) is not block decomposable, then the block-refined resource \(\tilde{m} \in BV(\Omega)\) of \(m\) satisfies that \(0 \leq \tilde{m} \leq 1\),
    \begin{equation}
        \int_\Omega \theta_{m, \mu} = \int_\Omega \theta_{\tilde{m}, \mu},
    \end{equation}
    \begin{equation}
        \int_\Omega m > \int_\Omega \tilde{m},
    \end{equation}
    \begin{equation}
        \int_\Omega \chi_{\{\tilde m = 1\}} - \int_\Omega \chi_{\{ m = 1\}} < 0,
    \end{equation}
    and
    \begin{equation}
        \int_\Omega \chi_{\{\tilde m = 0\}} - \int_\Omega \chi_{\{ m = 0\}} < 0.
    \end{equation}
\end{corollary}

\section{Optimal Control of Small Resources}

In this section, we discuss the optimal control under the small resource condition. Let us recall the definition of the advantage function. We consider a weak solution \(\theta_{l, b}\) of the following equation.
\begin{equation}
    \begin{cases}
        \theta_{l, b}'' + \theta_{l, b}(\chi_{(0, b)} - \theta_{l, b}) = 0 & \text{in } (0, l)
        \\
        \theta_{l, b}' = 0 & \text{on } 0, l.
    \end{cases}
\end{equation}

Then, the advantage function \(H(l, b)\) is defined as follows.
\begin{equation}
    H(l, b) = \int_{(0, l)} \theta_{l, b} - b.
\end{equation}

Note that the advantage function can be written as
\begin{equation}
    H(l, b) = lF_{\frac{1}{l^2}}(\frac{b}{l}) - b.
\end{equation}

Next, let us introduce the following lemma.

\begin{lemma}\label{analytic}
    Let \(\Omega = (0, 1)\) and \(m = \chi_{(0, m_0)}\). For any \(\mu_1 > 0\), there exists \(m_1 > 0\) such that we can write \(\theta_{m, \mu}\) as
    \begin{equation}
        \theta_{m, \mu} = m_0 + \sum_{k=1}^\infty\frac{\eta_{k, m}}{\mu^k},
    \end{equation}
    for any \((m_0, 1/\mu) \in (0, m_1)\times(0, 1/\mu_1)\), and \(F_\mu(m)\) is analytic of variables \(1/\mu\) and \(m_0\) on \((0, m_1)\times(0, 1/\mu_1)\).
\end{lemma}

Mazari, Nadin, and Privat proved in \cite{MNP20} that \(\theta_{m, \mu}\) is described as a power series of \(\mu\) for fixed \(m_0\) and large enough \(\mu\). We prove Lemma \ref{analytic} by applying the method introduced in \cite{MNP20}.

\begin{proof}
    First, we will construct a power series satisfying
    \begin{equation}\label{exp}
        \theta_{m, \mu} = m_0 + \sum_{k = 1}^\infty \frac{\eta_{k, m}}{\mu^k}
    \end{equation}
    and find its radius of convergence. In order to construct the series, we will identify each \(\eta_{k, m}\) and find an appropriate sequence \(\{\alpha(k)\}\) such that \(\|\eta_{k, m}\|_{L^\infty (\Omega)} \leq \alpha(k)\). Plugging this expansion into the equation of \(\theta_{m, \mu}\) and identifying at order \(\frac{1}{\mu}\), we can observe that \(\eta_{1, m}\) satisfies
    \begin{equation}
        \begin{cases}
            \eta_{1, m}'' + m_0(m - m_0) = 0 & \text{in } \Omega,
            \\
            \eta_{1, m} = 0 & \text{on } \partial\Omega.
        \end{cases}
    \end{equation}
    Note that \(\eta_{1, m}\) can be written as
    \begin{equation}
        \eta_{1, m} = \zeta_{1, m} + \beta_{1, m},
    \end{equation}
    where \(\zeta_{1, m}\) is the unique weak solution of
    \begin{equation}
        \begin{cases}
            \zeta_{1, m}'' + m_0(m - m_0) = 0 & \text{in } \Omega,
            \\
            \zeta_{1, m}' = 0 & \text{on } \partial\Omega,
            \\
            \int_\Omega \zeta_{1, m} = 0.
        \end{cases}
    \end{equation}
    To identify the constant \(\beta_{1, m}\), we will use the fact that
    \begin{equation}
        \int_\Omega \theta_{m, \mu}(m - \theta_{m, \mu}) = 0.
    \end{equation}
    Plugging (\ref{exp}) into the equation of \(\theta_{m, \mu}\) and identifying at order \(\frac{1}{\mu}\), we get
    \begin{equation}
        \beta_{1, m} = \frac{1}{m_0}\int_\Omega \zeta_{1, m}(m - m_0).
    \end{equation}
    Inductively, applying this process at order \(\frac{1}{\mu^k}\) for \(k \in \mathbb{N}\), we get \(\eta_{k, m} = \zeta_{k, m} + \beta_{k, m}\) where \(\zeta_{k, m}\) is the unique weak solution of
    \begin{equation}\label{eqeta}
        \begin{cases}
            \zeta_{1, m}'' + m_0(m - m_0) = 0 \text{ in } \Omega,
            \\
            \zeta_{2, m}'' + \eta_{1, m}(m - 2m_0) = 0 \text{ in } \Omega,
            \\
            \zeta''_{k + 1, m} + (m - 2m_0)\eta_{k, m} - \sum_{l = 1}^{k - 1} \eta_{l, m}\eta_{k - l, m} = 0 \text{ in } \Omega & \text{if } k \geq 2,
            \\
            \zeta'_{k, m} = 0 \text{ on } \partial\Omega,
            \\
            \int_\Omega \zeta_{k, m} = 0,
        \end{cases}
    \end{equation}
    and the constant \(\beta_{k, m}\) satisfies
    \begin{equation}\label{eqbeta}
        \begin{cases}
            \beta_{1, m} = \frac{1}{m_0}\int_\Omega \zeta_{1, m}(m - m_0),
            \\
            \beta_{2, m} = \frac{1}{m_0}\int_\Omega m\zeta_{2, m} - \frac{1}{m_0}\int_\Omega \eta_{1, m}^2,
            \\
            \beta_{k+1, m} = \frac{1}{m_0}\int_\Omega m\zeta_{k + 1, m} - \frac{1}{m_0}\sum_{l = 1}^k \int_\Omega \eta_{l, m}\eta_{k + 1 - l, m} & \text{if } k \geq 2.
        \end{cases}
    \end{equation}
    In addition, let us denote \(\zeta_{0, m} \equiv 0\) and \(\beta_{0, m} = m_0\) for convenience. Then, by applying \(W^{2, p}\) estimation, we obtain
    \begin{equation}
        \|\eta_{k + 1, m}\|_{W^{2, p}(\Omega)} \lesssim \|\eta_{k + 1, m}\|_{L^p (\Omega)} + \left\|(m - 2m_0)\eta_{k, m} - \sum_{l = 1}^{k - 1} \eta_{l, m}\eta_{k - l, m}\right\|_{L^p(\Omega)}.
    \end{equation}

    Now, we will find \(\alpha(k)\) satisfying \(\|\eta_{k, m}\|_{L^\infty(\Omega)} \leq \alpha(k)\) inductively. Observe that,
    \begin{equation}
        \|\eta_{1, m}\|_{L^\infty(\Omega)} \leq \|\zeta_{1, m}\|_{L^\infty(\Omega)} + |\beta_{1, m}| \lesssim \|\zeta_{1, m}\|_{L^\infty(\Omega)} \lesssim m_0^2.
    \end{equation}
    Hence, we may let \(\alpha(1) = m_0\). Now, assume that there exist such \(\alpha(\cdot)\) for \(1, \cdots, k\). Then we have
    \begin{equation}
        \left\|(m - 2m_0)\eta_{k, m} - \sum_{l = 0}^{k - 1} \eta_{l, m}\eta_{k - l, m}\right\|_{L^p(\Omega)} \lesssim \alpha(k) + \sum_{l = 0}^{k - 1} \alpha(l)\alpha(k - l),
    \end{equation}
    where \(\alpha(0) = m_0\). It implies that \(\eta_{k + 1, m} \in L^\infty(\Omega)\). Moreover, note that
    \begin{equation}
        \|\eta_{k + 1, m}\|_{L^p (\Omega)} \leq \|\zeta_{k + 1, m}\|_{L^p(\Omega)} + |\beta_{k + 1, m}|
    \end{equation}
    by the definition, and then applying the Poincaré inequality, we get
    \begin{equation}
        \|\zeta_{k + 1, m}\|_{L^p(\Omega)} \lesssim \|\zeta_{k + 1, m}'\|_{L^p(\Omega)}.
    \end{equation}
    Now, since \(\eta_{k + 1, m}' = \int_0^x \eta_{k+1, m}''(t)dt\), we have
    \begin{equation}
        \|\eta_{k + 1, m}'\|_{L^\infty (\Omega)} \lesssim \|\eta_{k, m}\|_{L^\infty(\Omega)} + \left\|\sum_{l = 0}^k \eta_{l, m}\eta_{k - l, m}\right\|_{L^\infty(\Omega)} \lesssim \alpha(k) + \sum_{l = 1}^{k-1} \alpha(l)\alpha(k - l).
    \end{equation}
    Moreover, in the same way,
    \begin{equation}
        |\beta_{k + 1, m}| \lesssim \alpha(k) + \frac{1}{m_0}\sum_{l = 1}^{k-1} \alpha(l)\alpha(k - l) + \frac{1}{m_0}\sum_{l = 1}^k \alpha(l)\alpha(k + 1 - l).
    \end{equation}
    Without loss of generality, we may assume that \(\alpha(k)\) is increasing in \(k\), then we get
    \begin{equation}
        |\beta_{k + 1, m}| \lesssim \frac{1}{m_0}\sum_{l = 1}^{k} \alpha(l)\alpha(k + 1 - l).
    \end{equation}
    Therefore, we get
    \begin{equation}
        \|\eta_{k + 1, m}\|_{L^\infty(\Omega)} \leq |\beta_{k + 1, m}| + \|\zeta_{k + 1, m}\|_{L^\infty (\Omega)} \lesssim \frac{1}{m_0} \sum_{l = 1}^{k} \alpha(l)\alpha(k + 1 - l).
    \end{equation}
    This inequality guarantees that we can construct \(\alpha(k)\) as
    \begin{equation}
        \alpha(1) = C_0m_0^2, \alpha(k + 1) = \frac{C_1}{m_0} \sum_{l = 1}^{k} \alpha(l)\alpha(k + 1 - l)
    \end{equation}
    for some constant \(C_0, C_1 > 0\). Then we can observe that the sequence \(\{\alpha(k)/m_0^{k+1}\}\) has a form of Catalan sequence which has an order at most \((4C_1)^kk^{-\frac{3}{2}}\). Therefore, the power series \(\sum \alpha(k)x^k\) has a radius of convergence at least \(\frac{1}{4m_0C_1}\), which implies that, by uniqueness of \(\theta_{m, \mu}\), there exists \(m_1 > 0\) such that
    \begin{equation}
        \theta_{m, \mu} = m_0 + \sum_{k = 1}^\infty \frac{\eta_{k, m}}{\mu^k}
    \end{equation}
    holds for \((m_0, 1/\mu) \in (0, m_1) \times (0, 1/\mu_1)\).

    Now, by the Fubini's theorem, we have
    \begin{equation}
        F_\mu(m) = \int_\Omega \theta_{m, \mu} = m_0 + \sum_{k = 1}^\infty \frac{1}{\mu^k}\int_\Omega \eta_{k, m}
    \end{equation}
    for \((m_0, 1/\mu) \in (0, m_1) \times (0, 1/\mu_1)\). To conclude that it is analytic, it suffices to prove that \(\int_\Omega \eta_{k, m}\) can be written as a polynomial of \(m_0\), and that the series is absolutely convergent.

    Inductively, it is easy to see that
    \begin{equation}
    \eta_{k, m} = 
        \begin{cases}
            m_0^{k+1}p_k(x, m_0) & \textnormal{if } x\in (0, m_0),
            \\
            m_0^{k+1}q_k(x, m_0) & \textnormal{if } x\in (m_0, 1).
        \end{cases}
    \end{equation}
    for some polynomial \(p_k, q_k\) of variables \(x\) and \(m_0\). Thus, we have
    \begin{equation}
        \int_\Omega \eta_{k, m} = \sum_{n = k+1}^\infty a_{k, n} m_0^n.
    \end{equation}
    where \(a_{k, n}\) is zero except for finitely many indices.
    
    Now, let us define \(\nu: \mathcal{P}(x, m_0) \to \mathbb{R}_{\geq0}\) as
    \begin{equation}
        \nu\left(\sum_{i, j = 0}^\infty a_{i, j}x^i m_0^j\right) = \sum_{i, j = 0}^\infty \left|a_{i, j}\right|.
    \end{equation}
    In addition, let us consider when \(p(x, m_0)\) is piecewise polynomial on \(\Omega\), which means that there exists a set \(\{I_1, I_2, \cdots, I_t\}\) of almost disjoint intervals such that \(\Omega = \cup_{i = 1}^t I_i\) and \(p(x, m_0) = p_i(x, m_0)\) for \(p_i \in \mathcal{P}(x, m_0)\) \(x \in I_i\). We can expand \(\nu\) as
    \begin{equation}
        \nu(p) = \max\{\nu(p_1), \nu(p_2), \cdots, \nu(p_t)\}.
    \end{equation}
    
    Now, let us find a sequence \(\{\gamma(k)\}\) such that
    \begin{equation}
        \nu(\eta_{k, m}) \leq \gamma(k)
    \end{equation}
    for \(k = 1, 2, \cdots\), so that
    \begin{equation}
        \sum_{n = k+1}^\infty |a_{k, n}| \leq 3\gamma(k),
    \end{equation}
    and
    \begin{equation}
        \sum_{n = k+1}^\infty |a_{k, n}| m_0^n \leq 3m_0^{k+1}\gamma(k).
    \end{equation}

    By (\ref{eqeta}), we have
    \begin{equation}
        \nu(\zeta_{k+1, m}'') \leq 3\gamma(k) + \sum_{l = 1}^{k-1} \gamma(l)\gamma(k-l).
    \end{equation}
    Consequently,
    \begin{equation}
        \nu(\zeta_{k+1, m}') \leq 6\gamma(k) + 2\sum_{l = 1}^{k-1} \gamma(l)\gamma(k-l),
    \end{equation}
    and we can write \(\zeta_{k + 1, m}\) as
    \begin{equation}
        \zeta_{k + 1, m} =
        \begin{cases}
            r_{k + 1}^{(1)}(m_0) + s_{k+1}^{(1)}(x, m_0) & \textnormal{if } x\in (0, m_0)
            \\
            r_{k + 1}^{(2)}(m_0) + s_{k+1}^{(2)}(x, m_0) & \textnormal{if } x\in (m_0, 1)
        \end{cases}
    \end{equation}
    where \(r_{k + 1}^{(1)}, r_{k + 1}^{(2)}, s_{k + 1}^{(1)}, s_{k + 1}^{(2)}\) are polynomials such that
    \begin{equation}
        \nu(s_{k + 1}^{(1)}) \leq 6\gamma(k) + 2\sum_{l = 1}^{k-1} \gamma(l)\gamma(k-l),
    \end{equation}
    and
    \begin{equation}
        \nu(s_{k + 1}^{(2)}) \leq 12\gamma(k) + 4\sum_{l = 1}^{k-1} \gamma(l)\gamma(k-l).
    \end{equation}

    Let us find \(\nu(r_{k + 1}^{(1)})\) and \(\nu(r_{k + 1}^{(2)})\). As \(\zeta_{k+1, m}\) is continuous, we have
    \begin{equation}
        r_{k+1}^{(1)}(m_0) = r_{k+1}^{(2)}(m_0) + s_{k+1}^{(2)}(m_0, m_0) - s_{k+1}^{(1)}(m_0, m_0),
    \end{equation}
    so
    \begin{equation}
        \nu(r_{k + 1}^{(1)}) \leq \nu(r_{k + 1}^{(2)}) + 18\gamma(k) + 6\sum_{l = 1}^{k-1} \gamma(l)\gamma(k-l).
    \end{equation}
    Moreover, as \(\int_\Omega \zeta_{k + 1, m} = 0\), we have
    \begin{equation}
        \begin{split}
            & m_0r_{k+1}^{(1)}(m_0) + (1-m_0)r_{k+1}^{(2)}(m_0)
            \\
            = & r_{k+1}^{(2)}(m_0) + m_0s_{k+1}^{(2)}(m_0, m_0) - m_0s_{k+1}^{(1)}(m_0, m_0)
            \\
            = & -\int_0^{m_0}s_{k+1}^{(1)}(x, m_0)dx - \int_{m_0}^{1}s_{k+1}^{(2)}(x, m_0)dx,
        \end{split}
    \end{equation}
    so we get
    \begin{equation}
        \nu(s_{k+1}^{(1)}) \leq 48\gamma(k) + 16\sum_{l = 1}^{k-1} \gamma(l)\gamma(k-l),
    \end{equation}
    and
    \begin{equation}
        \nu(s_{k+1}^{(2)}) \leq 66\gamma(k) + 22\sum_{l = 1}^{k-1} \gamma(l)\gamma(k-l).
    \end{equation}
    Therefore,
    \begin{equation}
        \nu(\zeta_{k+1, m}) \leq 72\gamma(k) + 24\sum_{l = 1}^{k-1} \gamma(l)\gamma(k-l).
    \end{equation}

    Plugging this result to (\ref{eqbeta}), we get
    \begin{equation}
        \nu(\beta_{k + 1, m}) \leq 72\gamma(k) + 26\sum_{l = 1}^{k-1} \gamma(l)\gamma(k-l).
    \end{equation}

    Finally, as \(\eta_{k+1, m} = \zeta_{k + 1, m} + \beta_{k + 1, m}\) by the definition, we get
    \begin{equation}
        \nu(\eta_{k+1, m}) \leq 144\gamma(k) + 50\sum_{l = 1}^{k-1} \gamma(l)\gamma(k-l),
    \end{equation}
    so we may construct \(\gamma(k)\) as
    \begin{equation}
        \gamma(k + 1) = 144\gamma(k) + 50\sum_{l = 1}^{k-1} \gamma(l)\gamma(k-l),
    \end{equation}
    as desired.

    We observe that it has a form of Catalan sequence which has a order at most \((200)^kk^{-\frac{3}{2}}\). Hence, the series
    \begin{equation}
        \sum_{k = 1}^{\infty} \frac{1}{\mu^k} \int_\Omega \eta_{k, m}
    \end{equation}
    of variable \(\mu\) and \(m_0\) absolutely converges if \(m_0 < \mu_1/200\).
\end{proof}

Using this lemma, we can directly calculate the first derivative of the total population.

\begin{corollary}
    Let \(\Omega = (0, 1)\) and \(m = \chi_{(0, m_0)}\). For any \(\mu_1 > 0\), there exist \(m_1 > 0\) such that,
    \begin{equation}
        \frac{\partial}{\partial m_0}F_\mu(m) = 1 + \frac{2}{3\mu}m_0 + O\left(\frac{m_0^2}{\mu}\right),
    \end{equation}
    and
    \begin{equation}
        \frac{\partial}{\partial \mu}F_\mu(m) = - \frac{1}{3\mu^2}m_0^2 + O\left(\frac{m_0^3}{\mu^2}\right).
    \end{equation}
    for any \((m_0, 1/\mu) \in (0, m_1)\times(0, 1/\mu_1)\).
\end{corollary}

Using the corollary above, let us prove Theorem \ref{optimal}.

\begin{proof}[Proof of Theorem \ref{optimal}]
    For any block decomposable \(m \in \mathcal{M}(\Omega) \cap BV(\Omega)\), let us denote the block decomposition of \(\theta_{m, \mu}\) as \(\theta_{l_1, b_1}, \theta_{l_2, b_2}, \cdots, \theta_{l_r, b_r}\). First, we will prove that
    \begin{equation}
        H(l_1, b_1) + H(l_2, b_2) < H(l_1 + l_2, b_1 + b_2)
    \end{equation}
    holds if \(m_0\) is sufficiently small. In other words, we will prove the superlinearity of the advantage function.

    Without loss of generality, we may suppose that \(l_1 \geq l_2\). Next, fix any \(C_0 \in (0, 1)\) and suppose that
    \begin{equation}
        \frac{l_2}{l_1} \geq C_0.
    \end{equation}

    By the definition of the advantage function and the corollary, we have
    \begin{equation}
        H(l_i, b_i) = \frac{1}{3}l_ib_i^2 + O(b_i^3) = \frac{1}{3}l_i^3a_i^2 + O(l_i^3a_i^3)
    \end{equation}
    where \(a_i := l_i/b_i\) for \(i = 1, 2\).
    
    Note that, by applying the Lemma \ref{bound}, we can find constants \(C_1, C_2 > 0\) such that \(C_1 m_0 \leq a_i \leq C_2 m_0\) for any \(i = 1, 2, \cdots, r\). Therefore,
    \begin{equation}
        \begin{split}
            & H(l_1 + l_2, b_1 + b_2) - H(l_1, b_1) - H(l_2, b_2) 
            \\
            = &\frac{2}{3}(l_1 + l_2)b_1b_2 + \frac{1}{3}l_1b_2^2 + \frac{1}{3}l_2b_1^2 + O(l_1^3\max\{a_1, a_2\}^3)
        \end{split}
    \end{equation}
    where \(a_i = \theta_{l_i, b_i}(0)\) for \(i = 1, 2\).

    Hence, it is enough to show that
    \begin{equation}
        \frac{1}{3}l_2b_1^2 + O(l_1^3\max\{a_1, a_2\}^3) = \frac{1}{3}l_2l_1^2a_1^2 + O(l_1^3\max\{a_1, a_2\}^3)
    \end{equation}
    is positive for small enough \(m_0\).

    Recall that we have \(C_1m_0 \leq a_1, a_2 \leq C_2m_0\), so
    \begin{equation}
        \begin{split}
            \frac{1}{3}l_2l_1^2a_1^2 + O(l_1^3\max\{a_1, a_2\}^3)
            \geq l_1^3m_0^2\left(\frac{C_0C_1}{3} + O(m_0)\right) \geq 0
        \end{split}
    \end{equation}
    for small enough \(m_0\). Moreover, for sufficiently small \(m_0\), we get
    \begin{equation}\label{ineq1}
        H(l_1 + l_2, b_1 + b_2) - H(l_1, b_1) - H(l_2, b_2) > \frac{C_0C_1}{6}l_1^3m_0^2 \geq \frac{C_0C_1}{48}(l_1 + l_2)^3m_0^2.
    \end{equation}

    Now, suppose that
    \begin{equation}
        \frac{l_2}{l_1} \leq C_0.
    \end{equation}
    
    Then, by applying the mean value theorem, we get
    \begin{equation}
        \begin{split}
            & H(l_1 + l_2, b_1 + b_2) - H(l_1, b_2) - H(l_2, b_2)
            \\
            = & \left.\frac{\partial H}{\partial l}\right|_{(l_3, b_1)}l_3 + \left.\frac{\partial H}{\partial l}\right|_{(l_1 + l_2, b_3)}b_2 - \frac{1}{3}l_2^3a_2^2 + O(l_2^3a_2^3)
        \end{split}
    \end{equation}
    for some \(l_3 \in (l_1, l_1 + l_2)\) and \(b_3 \in (b_1, b_1 + b_2)\).

    Note that, by simply calculation, we have
    \begin{equation}
        \frac{\partial}{\partial l}H (l, b) = \frac{1}{3}l^2a^2 + O(l^2a^3) = \frac{1}{3}b^2 + O(\frac{b^3}{l})
    \end{equation}
    and
    \begin{equation}
        \frac{\partial}{\partial l}H (l, b) = \frac{2}{3}l^2a + O(l^2a^2) = \frac{2}{3}lb + O(b^2).
    \end{equation}

    Hence,
    \begin{equation}
        \begin{split}
            & \left.\frac{\partial H}{\partial l}\right|_{(l_3, b_1)}l_2 + \left.\frac{\partial H}{\partial l}\right|_{(l_1 + l_2, b_3)}b_2
            \\
            = & \frac{1}{3}l_2b_1^2 + \frac{2}{3}(l_1 + l_2)b_2b_3 + O(l_2\frac{b_1^3}{l_1} + b_2^3).
        \end{split}
    \end{equation}

    Therefore, it suffices to show that
    \begin{equation}\label{result}
        \begin{split}
            & \frac{1}{3}l_2b_1^2 + \frac{2}{3}l_1b_1b_2 - \frac{1}{3}l_2^3a_2^2 + O(l_2\frac{b_1^3}{l_1} + b_2^3)
            \\
            = & \frac{1}{3}l_1^2l_2a_1^2 + \frac{2}{3}l_1^2l_2a_1a_2 - \frac{1}{3}l_2^3a_2^2 + O(l_1^2l_2m_0^3)
        \end{split}
    \end{equation}
    is positive for small enough \(m_0\). Observe that
    \begin{equation}
        \begin{split}
            & \frac{1}{3}l_1^2l_2a_1^2 + \frac{2}{3}l_1^2l_2a_1a_2 - \frac{1}{3}l_2^3a_2^2 + O(l_1^2l_2m_0^3).
            \\
            \geq & l_1^2l_2m_0^2\left(C_1^2 - \frac{C_0^2C_2^2}{3} + O(m_0)\right).
        \end{split}
    \end{equation}
    Let us set \(C_0\) small enough so that
    \begin{equation}
        C_1^2 - \frac{C_0^2C_2^2}{3} > 0.
    \end{equation}
    Then, for small enough \(m_0\), the desired result (\ref{result}) follows. Moreover, for sufficiently small \(m_0\), we get
    \begin{equation}\label{ineq2}
        \begin{split}
            H(l_1 + l_2, b_1 + b_2) - H(l_1, b_1) - H(l_2, b_2) > & \frac{3C_1^2 - C_0C_2^2}{6}l_1^2l_2m_0^2
            \\
            \geq & \frac{3C_1^2 - C_0C_2^2}{24}(l_1 + l_2)^2l_2m_0^2.
        \end{split}
    \end{equation}

    Now, let \(l_{max} := \max\{l_1, l_2, \cdots, l_r\} = l_k\) where \(k \in \{1, 2, \cdots, r\}\). In other words, let the size of the largest block be \(l_{max}\). Observe that
    \begin{equation}
        \sum_{i = 1}^r H(l_i, b_i) \leq \sum_{i = 1}^r l_i^3 a_i^2 \leq C_2^3m_0^3 \sum_{i = 1}^r l_i^3 \leq C_2^3m_0^3 |\Omega| l_{max}^2,
    \end{equation}
    so
    \begin{equation}\label{ineq3}
        F_\mu(\chi_{(0, b)}(x)) - F_\mu(m) \geq H(|\Omega|, |\Omega|m_0) - C_2^3m_0^3 |\Omega| l_{max}^2.
    \end{equation}

    By combining (\ref{ineq1}), (\ref{ineq2}), and (\ref{ineq3}), we observe that there exists small enough constant \(C_3 > 0\) such that
    \begin{equation}\label{block}
        F_\mu(\chi_{(0, b)}(x)) - F_\mu(m) \geq C_3m_0^2(|\Omega| - l_{max}).
    \end{equation}

    Now, let us consider any bang-bang type \(m \in \mathcal{M}(\Omega)\cap BV(\Omega)\) which is not block decomposable. Then, we have a block-refined resource \(\tilde{m} \in L^\infty(\Omega)\) of \(m\). By Corollary \ref{blockrefined}, we have
    \begin{equation}
        \int_\Omega \theta_{m, \mu} = \int_\Omega \theta_{\tilde{m}, \mu}
    \end{equation}
    and
    \begin{equation}
        \int_\Omega \tilde{m} < \int_\Omega m.
    \end{equation}
    
    By the definition of block-refined resources, we have a partition \(0 = x_0 < x_1 < \cdots < x_{r + 1} = |\Omega|\) of \(\Omega\) such that the restriction \(\tilde{m}_i\) of \(\tilde{m}\) on \(\Omega_i = (x_{i - 1}, x_i)\) becomes a monotone characteristic function or constant function. By rearranging the index, we may assume that \(\tilde{m}_1,\tilde{m}_2, \cdots,\tilde{m}_s\) are monotone characteristic function on \(\Omega_1, \Omega_2, \cdots, \Omega_s\) respectively, and \(\tilde{m}_{s+1},\tilde{m}_{s+2}, \cdots,\tilde{m}_r\) are constant function on \(\Omega_{s+1}, \Omega_{s+2}, \cdots, \Omega_r\) respectively.

    Let us write \(\tilde{l}_i := |\Omega_i|\) and \(\tilde{b}_i := \int_{\Omega_i}\tilde{m}_i\) Then, we observe that
    \begin{equation}
        \begin{split}
            \int_\Omega \theta_{\tilde{m}, \mu} = & \int_\Omega\tilde{m} + \sum_{i=1}^s H(\tilde{l}_i,\tilde{b}_i)
            \\
            < & \int_\Omega\tilde{m} + \sum_{i=1}^r H(\tilde{l}_i,\tilde{b}_i)
            \\
            < & \int_\Omega\tilde{m} + H\left(|\Omega|, \int_\Omega\tilde{m}\right).
        \end{split}
    \end{equation}

    Writing \(\tilde{b} := \int_\Omega\tilde{m}\), we can observe that
    \begin{equation}
        \int_\Omega\tilde{m} + H\left(|\Omega|,\tilde{b}\right) = \int_\Omega \theta_{\chi_{(0,\tilde{b})}, \mu},
    \end{equation}
    so we get
    \begin{equation}
        \begin{split}
            F_\mu(m) \leq & \int_\Omega \theta_{\tilde{m}, \mu} < \int_\Omega \theta_{\chi_{(0,\tilde{b})}, \mu}
            \\
             < & \int_\Omega \theta_{\chi_{(0, b)}, \mu} = F_\mu(\chi_{(0, b)}(x)) = F_\mu(\chi_{(0, b)}(|\Omega| - x)).
        \end{split}
    \end{equation}
    
    Now, let \(\tilde{l}_{max} := \max\{\tilde{l}_1,\tilde{l}_2, \cdots,\tilde{l}_r\} =\tilde{l}_k\) where \(k \in \{1, 2, \cdots, r\}\). In other words, we suppose that \(\Omega_k =: \Omega_{max}\) is the largest block among \(\Omega_1, \Omega_2, \cdots, \Omega_r\). Observe that,
    \begin{equation}
        \begin{split}
            \int_\Omega \theta_{m, \mu} = & \int_\Omega \theta_{\tilde{m}, \mu}
            \\
            < & \int_\Omega\tilde{m} + \sum_{i = 1}^rH(\tilde{l}_i,\tilde{b}_i)
            \\
            \leq & \int_\Omega\tilde{m} + H(\tilde{l}_{max},\tilde{b}_{max}) + H\left(|\Omega| -\tilde{l}_{max}, \int_\Omega\tilde{m} -\tilde{b}_{max}\right).
        \end{split}
    \end{equation}
    where \(\tilde{b}_{max} :=\tilde{b}_k\)

    By (\ref{ineq1}), (\ref{ineq2}) and (\ref{ineq3}), we can make the constant \(C_3 > 0\) smaller, if needed, so that
    \begin{equation}\label{notblock}
        F_\mu(\chi_{(0, b)}(x)) - F_\mu(m) \geq C_3m_0^2(|\Omega| - \tilde{l}_{max})
    \end{equation}
    also holds.
    
    Finally, let us find every \(m_\mu^*\) by using (\ref{block}) and (\ref{notblock}). Note that we can find a sequence of bang-bang type resources \(\{m_n\} \subset \mathcal{M}(\Omega)\cap BV(\Omega)\) such that \(m_n \to m_\mu^*\) in \(L^2(\Omega)\), and \(\theta_{m_n, \mu} \to \theta_{m_\mu^*, \mu}\) in \(L^2(\Omega)\). Therefore,
    \begin{equation}
        F_{\mu}(m_\mu^*) = \lim_{n \to \infty} F_\mu (m_n) \leq F_\mu(\chi_{(0, b)}(x)) = F_\mu(\chi_{(0, b)}(|\Omega| - x)).
    \end{equation}
    Thus, we conclude that \(\chi_{(0, b)}(x)\) and \(\chi_{(0, b)}(|\Omega| - x)\) are optimal controls.

    Furthermore, let us prove that there are no other optimal controls. We may assume that, by taking a subsequence if it is necessary, \(\{m_n\}\) is either a sequence of block decomposable resources or a sequence of block in-decomposable resources. In the former case, by (\ref{block}), the size of the largest block must converge to \(|\Omega|\), which implies that \(m_n \to \chi_{(0, b)}(x)\) or \(\chi_{(0, b)}(|\Omega| - x)\) in \(L^2(\Omega)\).

    In the same way, in the latter case, we get \(\tilde{m}_n \to \chi_{(0, b)}(x)\) or \(\chi_{(0, b)}(|\Omega| - x)\) in \(L^2(\Omega)\). Therefore, it suffices to show that \(\tilde{m}_n - m_n \to 0\) in \(L^2(\Omega)\) as \(n \to \infty\).

    \begin{figure}[htb!]
    \centering
    \begin{tikzpicture}[scale = 3.4]
        \draw[thick, ->](-0.1, 0)--(1.1, 0) node [anchor=north]{\(\theta\)};
        \draw[thick, ->](0, -0.1)--(0, 0.9) node [anchor=east]{\(\theta'\)};
        \node[below] at (-0.2, 0) {\(0\)};
        \draw[domain=0.1:0.5, smooth,thick,,variable=\x,black] plot ({\x},{sqrt(0.64 - (\x - 0.9)*(\x - 0.9))});
        \draw[domain=0.5:0.9, smooth,thick,,variable=\x,black] plot ({\x},{sqrt(0.64 - (\x - 0.1)*(\x - 0.1))});
        \node[below] at (0.1, 0) {\(A\)};
        \node[below] at (0.9, 0) {\(C\)};
        \node[above] at (0.5, 0.7) {\(B\)};
        \draw[thick, blue](0.5, 0)--(0.5, 0.69);
        \draw[domain=0.3:0.68, smooth,thick,,variable=\x,red] plot ({\x},{sqrt(0.36 - (\x - 0.9)*(\x - 0.9))});
        \draw[domain=0.32:0.7, smooth,thick,,variable=\x,red] plot ({\x},{sqrt(0.36 - (\x - 0.1)*(\x - 0.1))});
        \node[below] at (0.5, 0) {\(D\)};
        \node[right] at (0.54, 0.45) {\(E\)};
        \node[below] at (0.3, 0) {\(F\)};
        \node[below] at (0.7, 0) {\(G\)};
        \node[above, left] at (0.32, 0.58) {\(H\)};
        \node[above, right] at (0.68, 0.58) {\(I\)};
    \end{tikzpicture}
    \caption{The Phase Diagram of \(\theta_{\tilde{m}_n, \mu}\)}
    \label{5}
    \end{figure}
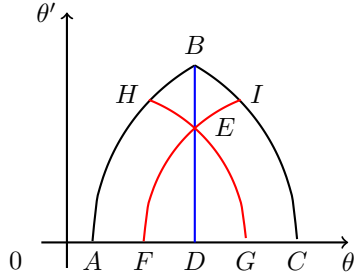
    
    Let us draw the phase diagram of the largest block of \(\tilde{m}_n\) as above. The black line is the phase diagram of \(\tilde{m}_n\), \(A\) is starting point which corresponds to \(x = \tilde{c}_0\), and \(C\) is the end point of the diagram which corresponds to \(x = \tilde{c}_2\). Let \(B\) be the critical point of the diagram which corresponds to \(x = \tilde{c}_1\). We may define the curve \((\theta_1, \theta_1')\) which continues \(\{(\theta_{\tilde{m}, \mu}(x), \theta_{\tilde{m}, \mu}'(x)) : x \in (\tilde{c}_0, \tilde{c}_1)\}\) by
    \begin{equation}
        \left\{ (\theta_1(x), \theta_1'(x)) : (\theta_1')^2 = \frac{2}{3}\theta_1^3 - \frac{2}{3}\theta^3(\tilde{c}_0) \right\}.
    \end{equation}
    Moreover, we may define the curve \((\theta_2, \theta_2')\) which continues \(\{(\theta_{\tilde{m}, \mu}(x), \theta_{\tilde{m}, \mu}'(x)) : x \in (\tilde{c}_1, \tilde{c}_2)\}\) by
    \begin{equation}
        \left\{ (\theta_2(x), \theta_2'(x)) : (\theta_2')^2 = \frac{2}{3}\theta_2^3 - \theta_2^2 - \left(\frac{2}{3}\theta^3(\tilde{c}_2) - \theta^2(\tilde{c}_2)\right) \right\}.
    \end{equation}
    Finally, note that \(\tilde{c}_2 - \tilde{c}_1 \to |\Omega|\) as \(n \to \infty\).
    
    Now, let \(D\) the foot of the perpendicular of \(B\) onto \(\theta\) axis. Take any point \(D\) on the line segment between \(B\) and \(D\) which is depicted as the blue line on the figure above. Next, let us define two curves \((\theta_3, \theta_3')\) and \((\theta_4, \theta_4')\) by
    \begin{equation}\label{cur1}
        \left\{ (\theta_3, \theta_3') : (\theta_3')^2 - (\theta'(E))^2 = \frac{2}{3}\theta_3^3 - \frac{2}{3}\theta^3(E) \right\}
    \end{equation}
    and
    \begin{equation}\label{cur2}
        \left\{ (\theta_4, \theta_4') : (\theta_4')^2 - (\theta'(E))^2 = \frac{2}{3}\theta_4^3 - \theta_4^2 - \left(\frac{2}{3}\theta^3(E) - \theta^2(E)\right) \right\}
    \end{equation}
    passing through \(E\). These two curves are depicted in the figure above as the red curves. \(F\) and \(G\) are intersection points between the \(\theta\)-axis and (\ref{cur1}), and the \(\theta\)-axis and (\ref{cur2}) respectively, and \(H\) and \(I\) are intersection points between black curve and (\ref{cur2}), and the black curve and (\ref{cur1}), respectively. Also, take \(h, i \in \Omega\) such that \(H = (\theta_1(h), \theta'_1(h))\), \(I = (\theta_2(i), \theta'_2(i))\).

    Then, we may assume that \(H = (\theta_3(h), \theta'_3(h))\) and \(I = (\theta_4(i), \theta'_4(i))\). Moreover, we can find \(e_-\) and \(e_+\) such that \(h \leq e_+ \leq e_- \leq i\) and \(E = (\theta_3(e_-), \theta'_3(e_-)) = (\theta_4(e_+), \theta_4'(e_+))\). It is easy to observe that \(B = E\) if \(e_- = e_+\), and \(e_- - e_+ \to 0\) implies that \(|B - E| \to 0\) regardless of \(n\).
    
    For any fixed \(n\), we observe that \(e_- - e_+\) increases as \(|B - E|\) increases. Hence, for large enough \(n\), we can find unique \(E\) such that
    \begin{equation}
        e_- - e_+ = 2(|\Omega| - \tilde{c}_2 + \tilde{c}_0).
    \end{equation}
    Then, as \(|\Omega| - \tilde{c}_2 + \tilde{c}_0\) as \(n \to \infty\), we have \(|E - B| \to 0\) as \(n \to \infty\).

    Let us recall that the phase diagram of \(\theta_{\tilde{m}_n, \mu}\) is obtained by modifying the part of the phase diagram of \(\theta_{m_n, \mu}\). We have local minimum and local maximum points \(c_0, c_2 \in \Omega\) of \(\theta_{m_n, \mu}\) such that \(c_0 < \tilde{c}_0 < \tilde{c}_2 < c_2\) and there are no other local extrema between \(c_0\) and \(c_2\). Then we easily observe that this phase diagram of \(\theta_{m_n, \mu}(x)\) for \(x \in {(c_0, c_2)}\) must lie between the black curve and the red curve which follows \(F \to E \to G\): if not, by slightly modifying the phase diagram of \(\theta_{m_n, \mu}\), we obtain the \(m'_n\) such that the phase diagram of \(\theta_{m_n', \mu}\) follows \(A \to H \to E \to I \to C\). By Corollary \ref{modif}, we have
    \begin{equation}
        \int_\Omega \chi_{\{\tilde{m}_n = 0, 1\}} \geq \int_\Omega \chi_{\{m_n' = 0, 1\}} \geq (e_- - \tilde{c}_0) + (\tilde{c}_2 - e_+) > |\Omega|,
    \end{equation}
    which is a contradiction.
    
    Hence, we can take points \(J\) and \(K\) as intersection points between the phase diagram of \(m_n\) and \(H \to E\), and the phase diagram of \(m_n\) and \(E \to I\), respectively. In addition, set \(j_-\) and \(k_+\) such that \(c_0 < j_- \leq k_+ < c_2\), \(J = (\theta_{\tilde{m}_n, \mu}(j_-), \theta'_{\tilde{m}_n, \mu}(j_-))\) and \(K = (\theta_{\tilde{m}_n, \mu}(k_+), \theta'_{\tilde{m}_n, \mu}(k_+))\). On the other hand, let \(j_+\) and \(k_-\) such that \(J = (\theta_4(j_+), \theta_4'(j_+))\) and \(K = (\theta_3(k_-), \theta_3'(k_-))\).

    \begin{figure}[htb!]
    \centering
    \begin{tikzpicture}[scale = 5]
        \draw[thick, ->](-0.1, 0)--(1.1, 0) node [anchor=north]{\(\theta\)};
        \draw[thick, ->](0, -0.1)--(0, 0.9) node [anchor=east]{\(\theta'\)};
        \node[below] at (-0.2, 0) {\(0\)};
        \node[below] at (0.1, 0) {\(A\)};
        \node[below] at (0.9, 0) {\(C\)};
        \node[above] at (0.5, 0.7) {\(B\)};
        \draw[domain=0.3:0.68, smooth,thick,dotted,variable=\x,red] plot ({\x},{sqrt(0.36 - (\x - 0.9)*(\x - 0.9))});
        \draw[domain=0.32:0.7, smooth,thick,dotted,variable=\x,red] plot ({\x},{sqrt(0.36 - (\x - 0.1)*(\x - 0.1))});
        \node[below] at (0.5, 0.41) {\(E\)};
        \node[below] at (0.3, 0) {\(F\)};
        \node[below] at (0.7, 0) {\(G\)};
        \node[above, left] at (0.32, 0.58) {\(H\)};
        \node[above, right] at (0.68, 0.58) {\(I\)};

        \draw[domain=0.1:0.26, smooth,thick,,variable=\x,black] plot ({\x},{sqrt(0.64 - (\x - 0.9)*(\x - 0.9))});
        \draw[domain=0.74:0.9, smooth,thick,,variable=\x,black] plot ({\x},{sqrt(0.64 - (\x - 0.1)*(\x - 0.1))});

        \draw[domain=0.26:0.35, smooth,thick,,variable=\x,black] plot ({\x},{sqrt(0.25 - (\x - 0.1)*(\x - 0.1))});
        \draw[domain=0.65:0.74, smooth,thick,,variable=\x,black] plot ({\x},{sqrt(0.25 - (\x - 0.9)*(\x - 0.9))});
        \draw[domain=0.35:0.5, smooth,thick,,variable=\x,black] plot ({\x},{sqrt(0.49 - (\x - 0.9)*(\x - 0.9))});
        \draw[domain=0.5:0.65, smooth,thick,,variable=\x,black] plot ({\x},{sqrt(0.49 - (\x - 0.1)*(\x - 0.1))});

        \draw[domain=0.26:0.5, smooth,thick,dotted,variable=\x,black] plot ({\x},{sqrt(0.64 - (\x - 0.9)*(\x - 0.9))});
        \draw[domain=0.5:0.74, smooth,thick,dotted,variable=\x,black] plot ({\x},{sqrt(0.64 - (\x - 0.1)*(\x - 0.1))});

        \filldraw[black] (0.42,0.51) circle (0.3pt);
        \node[above] at (0.42,0.52) {\(J\)};
        \filldraw[black] (0.58,0.51) circle (0.3pt);
        \node[above] at (0.58,0.52) {\(K\)};
    \end{tikzpicture}
    \caption{The Phase Diagram of \(\theta_{m_n, \mu}\)}
    \label{6}
    \end{figure}
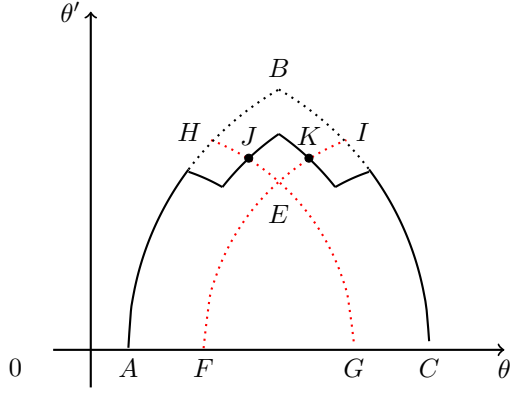

    Note that, by Corollary \ref{modif}, we have
    \begin{equation}
        0 \leq \int_{c_0}^{j_-} \chi_{\{m_n = 0\}} - (h - \tilde{c}_0)
    \end{equation}
    and
    \begin{equation}
        0 \leq \int_{c_0}^{j_-} \chi_{\{m_n = 1\}} - (j_+ - h).
    \end{equation}
    Now, by slightly modifying the phase diagram of \(\theta_{m_n, \mu}\), we can obtain \(m_n''\) such that the phase diagram of \(\theta_{m_n'', \mu}\) follows \(A \to H \to J \to K \to C\). Then,
    \begin{equation}
        |\Omega| - \int_\Omega \chi_{\{m_n'' = 0, 1\}} = \left(\int_{c_0}^{j_-} \chi_{\{m_n = 0\}} - (h - \tilde{c}_0)\right) + \left(\int_{c_0}^{j_-} \chi_{\{m_n = 1\}} - (j_+ - h)\right)
    \end{equation}
    and
    \begin{equation}
        \int_\Omega \chi_{\{m_n'' = 0, 1\}} \geq \int_\Omega \chi_{\{\tilde{m}_n = 0, 1\}} \geq \tilde{c}_2 - \tilde{c}_0
    \end{equation}
    follow by Corollary \ref{modif}. Therefore, we get
    \begin{equation}
        \int_{c_0}^{j_-} \chi_{\{m_n = 0\}} - (h - \tilde{c}_0) \leq |\Omega| - \tilde{c}_0 + \tilde{c}_2
    \end{equation}
    and
    \begin{equation}
        \int_{c_0}^{j_-} \chi_{\{m_n = 1\}} - (j_+ - h) \leq |\Omega| - \tilde{c}_0 + \tilde{c}_2.
    \end{equation}
    Furthermore, in the same way, we get
    \begin{equation}
        0 \leq \int_{k_+}^{c_2} \chi_{\{m_n = 0\}} - (i - k_-) \leq |\Omega| - \tilde{c}_0 + \tilde{c}_2
    \end{equation}
    and
    \begin{equation}
        0 \leq \int_{k_+}^{c_2} \chi_{\{m_n = 1\}} - (\tilde{c}_2 - i) \leq |\Omega| - \tilde{c}_0 + \tilde{c}_2.
    \end{equation}
    Finally, it is easy to observe that
    \begin{equation}
        0 \leq k_+ - j_- \leq (b_- - h) - (i - b_+).
    \end{equation}

    Recall that \(|B - E| \to 0\) as \(n \to \infty\). Moreover, we can observe that, \(b_- - h, i - b_+, e_- - h, i - e_+ \to 0\) as \(|B - E| \to 0\). Therefore, we get
    \begin{equation}
        \begin{split}
            \int_{c_0}^{j_-} \chi_{\{m_n = 0\}} - (b_- - \tilde{c}_0), \int_{k_+}^{c_2} \chi_{\{m_n = 1\}} - (\tilde{c}_2 - b_+)
            \\
            \int_{c_0}^{j_-} \chi_{\{m_n = 1\}}, \int_{k_+}^{c_2} \chi_{\{m_n = 0\}} \quad \textnormal{and} \quad k_+ - j_- \to 0,
        \end{split}
    \end{equation}
    which implies that \(\tilde{m}_n - m_n \to 0\) in \(L^2(\Omega)\) as \(n \to \infty\).
\end{proof}

\section*{Acknowledgment}

This work is an outcome of the Undergraduate Research Program (URP) at KAIST. The authors thank Professor Jaeyoung Byeon for his helpful advice and sincere discussion.

\bibliographystyle{amsplain}

\bibliography{refs}

\end{document}